\documentclass[12pt]{article}
\usepackage{amsthm, amsmath, natbib}
\usepackage{enumerate}
\usepackage{psfrag,epsf}
\usepackage{graphicx}
\usepackage{multirow}
\usepackage{subcaption}
\usepackage{mathtools}
\usepackage{url} 
\usepackage{titlesec}
\usepackage{hhline}
\titlelabel{\thetitle.\quad}
\usepackage{titlesec}
\usepackage{titletoc}
\usepackage[title,titletoc]{appendix}

\newtheorem{thm}{Theorem}[section]
\newtheorem{defn}{Definition}[section]
\newtheorem{lem}{Lemma}[section]
\newtheorem{corollary}{Corollary}[section]

\newcommand{\blind}{1}

\begin{document}

\newcommand{\ssection}[1]{\section[#1]{\centering #1}}

\def\spacingset#1{\renewcommand{\baselinestretch}%
{#1}\small\normalsize} \spacingset{1}


\if1\blind
{
  \title{\bf Data-dependent Posterior Propriety of Bayesian Beta-Binomial-Logit Model}
  \author{Hyungsuk Tak\hspace{.2cm}    and     Carl N. Morris\thanks{Hyungsuk Tak is a PhD candidate and Carl N. Morris is an Emeritus Professor at the Department of Statistics at Harvard University, Cambridge, MA 02138 (e-mail: hyungsuk.tak@gmail.com and morris@stat.harvard.edu). The authors thank Joseph Kelly for discussions and Steven Finch for proofreading.} \\
    Department of Statistics\\Harvard University\\1 Oxford Street\\Cambridge, MA 02138}
  \maketitle
} \fi

\if0\blind
{
  \bigskip
  \bigskip
  \bigskip
  \begin{center}
    {\LARGE\bf Posterior Propriety and Frequency Coverage Evaluation of Bayesian Beta-Binomial-Logit Model}
\end{center}
  \medskip
} \fi

\bigskip
\begin{abstract}
A Beta-Binomial-Logit model is a Beta-Binomial model with covariate information incorporated via a logistic regression. Posterior propriety of a Bayesian Beta-Binomial-Logit model can be data-dependent for improper hyper-prior distributions. Various researchers in the literature have unknowingly used improper posterior distributions or have  given incorrect statements about posterior propriety because checking posterior propriety can be challenging due to the complicated functional form of a Beta-Binomial-Logit model. We derive data-dependent necessary and sufficient conditions for posterior propriety within a class of hyper-prior distributions that encompass those used in previous studies.

\end{abstract}

\noindent%
\vfill

\newpage
\spacingset{1.45} 

\ssection{INTRODUCTION}\label{sec:intro}

Binomial data from several independent groups sometimes have more variability than the assumed Binomial distribution. To account for this extra-Binomial variability, called overdispersion, a Beta-Binomial (BB) model \citep{skellam1948}  puts a conjugate Beta prior distribution on unknown success probabilities by treating them as random effects. A Beta-Binomial-Logit (BBL) model \citep{williams1982, kahn1996}  is one way to incorporate covariate information into the BB model. The BBL model has a two-level structure as follows: For each of $k$ independent groups ($j=1, 2, \ldots, k$),
\begin{align}
 y_j\mid p_j  &\stackrel{\textrm{indep.}}{\sim}  \textrm{Bin}(n_j,~ p_j),\label{first-data}\\
p_j\mid r, \boldsymbol{\beta} &\stackrel{\textrm{indep.}}{\sim} \textrm{Beta}( rp^E_j,~ r(1-p^E_j) ),\label{second-beta}\\
 p^E_j\equiv &~E(p_j\mid r, \boldsymbol{\beta})=\frac{\exp(\boldsymbol{x}^\top_j\boldsymbol{\beta})}{1+\exp(\boldsymbol{x}^\top_j\boldsymbol{\beta})}\label{logistic_reg}
\end{align}   
where $y_j$ is the number of  successful outcomes out of $n_j$ trials, a sufficient statistic for the random effect $p_j$, $p^E_j$ denotes the expected random effect, $\boldsymbol{x}_j=(x_{j1}, x_{j2}, \ldots, x_{jm})^\top$ is a covariate vector of length $m$  for group $j$,  $\boldsymbol{\beta}=(\beta_1, \beta_2, \ldots, \beta_m)^\top$ is an $m\times 1$  vector of logistic regression coefficients, and $r$ represents the amount of prior information \citep{albert1988computational} as $n_j$ indicates the amount of observed information in group $j$. When there is no covariate with only an intercept term, i.e., $\boldsymbol{x}_j^\top\boldsymbol{\beta}=\beta_1$, the conjugate Beta prior distribution in \eqref{second-beta} is exchangeable, and the BBL model reduces to the BB model. 

A Bayesian approach to the BBL model needs a joint hyper-prior distribution of $r$ and $\boldsymbol{\beta}$ that affects posterior propriety. Though a proper joint hyper-prior distribution guarantees posterior propriety, various researchers have used improper hyper-prior distributions hoping for minimal impact on the posterior inference. The articles of \cite{albert1988computational} and \cite{daniels1999prior} use a hyper-prior probability density function (PDF) of a proper uniform shrinkage prior for $r$, $dr/(1+r)^2$, and independently an improper flat hyper-prior PDF for $\boldsymbol{\beta}$, $d\boldsymbol{\beta}$. Chapter 5 of \cite{gelman2013bayesian}  suggests putting an improper hyper-prior PDF on $r$, $dr/r^{1.5}$, and independently a proper standard logistic distribution on $\beta_1$ when  $\boldsymbol{x}^\top\boldsymbol{\beta}=\beta_1$. (Their Chapter~5 uses a different parameterization: $p_j\mid \alpha, \beta \sim\textrm{Beta}(\alpha, \beta)$ and $d\alpha d\beta/(\alpha+\beta)^{2.5}$. Transforming $r=\alpha+\beta$ and $p^E=\alpha/(\alpha+\beta)$, we obtain $dp^Edr/r^{1.5}$.) However,  the paper of \cite{albert1988computational} does not address  posterior propriety, the proposition in \cite{daniels1999prior} incorrectly concludes that posterior propriety holds regardless of the data, and Chapter 5 of \cite{gelman2013bayesian} specifies an incorrect data-dependent condition for posterior propriety.

To illustrate with an overly simple example for data-dependent conditions for posterior propriety, we toss two biased coins twice each ($n_j=2$ for $j=1, 2$).  Let $y_j$ indicate the number of Heads for coin $j$, and assume a BB model with $\boldsymbol{x}^\top\boldsymbol{\beta}=\beta_1$. If we use any proper hyper-prior PDF for $r$ together with an improper flat density on an intercept term $\beta_1$ independently,  posterior propriety holds except when both coins land either all Heads ($y_1=y_2=2$) or all Tails ($y_1=y_2=0$) as shown by an X in the diagram. Here the notation O means that the resulting posterior is proper. See Section \ref{example_coin} for details.

\begin{table}[h!]
\begin{center}
\label{table_intro}
          \begin{tabular}{c|ccc}
          $y_1 \backslash y_2$ & 0 & 1 & 2 \\
          \hline
          0 & X & O & O \\
          1 & O & O & O \\
          2 & O & O &  X\\
          \end{tabular}   
          \end{center}
\end{table}

Also, there is a hyper-prior PDF for $r$ that always leads to an improper posterior distribution regardless of the data. The article of \citet{kass1989approximate} adopts an improper joint hyper-prior PDF, $d\boldsymbol{\beta}dr/r$, without addressing posterior propriety. The paper of \citet{kahn1996} uses the same improper hyper-prior PDF for $r$, $dr/r$, and independently a proper multivariate Gaussian hyper-prior PDF for $\boldsymbol{\beta}$, declaring posterior propriety without a proof.  However,  the hyper-prior PDF $dr/r$ used in both articles always leads to an improper posterior  regardless of the data.

We derive data-dependent necessary and sufficient conditions for  posterior propriety of a Bayesian BBL model equipped with various joint hyper-prior distributions in Table~\ref{posterior_condition1}, the centerpiece of this article. We mainly work on  a class of hyper-prior PDFs for $r$, $dr/(t+r)^{u+1}$, where $t$ is non-negative and $u$ is positive. It includes a proper $dr/(1+r)^2$ \citep{albert1988computational, daniels1999prior} and an improper $dr/r^{1.5}$ \citep{gelman2013bayesian} as special cases. Independently the hyper-prior PDF for $\boldsymbol{\beta}$ is improper flat (Lebesque measure)  for its intended minimal impact on posterior inference. We also consider any proper hyper-prior PDF for $r$ and that for $\boldsymbol{\beta}$ in proving conditions for posterior propriety. 





The article is organized as follows. We derive the equivalent inferential model of the Bayesian BBL model in Section \ref{sec2}. We derive data-dependent conditions for posterior propriety and address posterior propriety in past studies in Section \ref{sec3}. We check posterior propriety in two examples in Section \ref{sec4}. 


\ssection{INFERENTIAL MODEL}\label{sec2}

One advantage of the  BBL model is that it allows the shrinkage interpretation in inference \citep{kahn1996}. For $j=1, 2, \ldots, k$, the conditional posterior distribution of a random effect $p_j$ given hyper-parameters and data is 
\begin{equation}\label{post}
p_j\mid r, \boldsymbol{\beta}, \boldsymbol{y}\stackrel{\textrm{indep.}}{\sim} \textrm{Beta}( rp^E_j + y_j,~ r(1-p^E_j) +(n_j-y_j) )
\end{equation}
where $\boldsymbol{y}=(y_1, y_2, \ldots, y_k)^\top$. The  posterior mean of \eqref{post} is  $\hat{p}_j\equiv(1-B_j)\bar{y}_j + B_jp^E_j$, a convex combination of the observed proportion $\bar{y}_{j} = y_j/n_j$ and the expected random effect  $p^E_{j}$ weighted by the relative amount of information in the prior compared to the data, called a shrinkage factor $B_j=r / (r + n_j)$. If the conjugate prior distribution contains more information than the observed data, i.e., ensemble sample size $r$ exceeds individual sample size $n_{j}$, then the posterior mean shrinks more towards $p^E_j$ than towards $\bar{y}_j$. The posterior variance  of this conditional posterior distribution is a quadratic function of $\hat{p}_j$, i.e., $\hat{p}_j(1-\hat{p}_j)/(r+n_j+1)$.

The conjugate Beta prior distribution of random effects in \eqref{second-beta} has unknown hyper-parameters, $r$ and $\boldsymbol{\beta}$. Assuming $r$ and $\boldsymbol{\beta}$ are independent a priori, we introduce their joint hyper-prior PDF as follows:
\begin{equation}\label{hyper-prior}
\pi_{\textrm{hyp.prior}}(r, \boldsymbol{\beta})=f(r)g(\boldsymbol{\beta})\propto \frac{g(\boldsymbol{\beta})}{(t+r)^{u+1}},~\textrm{for } t\ge0 \textrm{ and } u>0.
\end{equation}
This class of  hyper-prior PDFs for $r$, i.e., $dr/(t+r)^{u+1}$, is proper if $t>0$  and improper if $t=0$. A hyper-prior PDF for a uniform shrinkage prior on $r$ is $dr/(t+r)^2$ with $u=1$  for any positive constant $t$ \citep{christiansen1997hierarchical}. A special case of the uniform shrinkage prior density function is $dr/(1+r)^2$ corresponding to $t=1$ used by \cite{albert1988computational} and \cite{daniels1999prior}. As $t$ goes to zero, a proper uniform shrinkage prior density becomes close to an improper hyper-prior PDF $dr/r^2$.  The improper hyper-prior PDF $dr/r^{1.5}$ suggested in Chapter 5 of \cite{gelman2013bayesian}   corresponds to  $u=0.5$ and $t=0$. The hyper-prior PDF $dr/r$ is not included in the class because it always leads to an improper posterior distribution regardless of the data; see Section \ref{post_proper_other_article}.  The  hyper-prior PDF for $\boldsymbol{\beta}$, $g(\boldsymbol{\beta})$, can be any proper PDF or an improper flat density.

If the symbol $A$ represents a second-level variance component in a two-level Gaussian multilevel model, e.g., $y_j\mid \mu_j\sim N(\mu_j, 1)$ and $\mu_j\mid A\sim N(0, A)$, then $A$ is proportional to $1/r$. The improper hyper-prior PDF $dr/r^{2}=-d(1/r)$ corresponds to $dA$  leading to Stein's harmonic prior \citep{morris2011}, $dr/r^{1.5}$ corresponds to $dA/\sqrt{A}$ \citep{gelman2013bayesian}, and $dr/r$ is equivalent to an inappropriate choice $dA/A$ \citep{morris2012shrinkage}.


The marginal distribution of the data follows independent Beta-Binomial distributions \citep{skellam1948} with random effects integrated out. The probability mass function for the Beta-Binomial distribution is, for $j=1, 2, \ldots, k$,
\begin{equation}\label{marginal_betabinom}
\pi_{\textrm{obs}}(y_j\mid r, \boldsymbol{\beta})=\binom{n_j}{y_j}\frac{B(y_j+rp^E_j, ~n_j-y_j+r(1-p^E_j))}{B(rp^E_j, ~r(1-p^E_j))}
\end{equation}
where the notation $B(a, b)$ indicates a beta function defined as $\int_0^1 v^{a-1}(1-v)^{b-1}dv$ for positive constants $a$ and $b$. This distribution depends on $\boldsymbol{\beta}$ because the expected random effects, $\{p^E_1, p^E_2, \ldots, p^E_k\}$, are a function of $\boldsymbol{\beta}$ as shown in \eqref{logistic_reg}. The likelihood function of $r$ and $\boldsymbol{\beta}$ is the product of these Beta-Binomial probability mass functions being treated as expressions in $r$ and $\boldsymbol{\beta}$, i.e., 
\begin{equation}\label{marginal}
L(r, \boldsymbol{\beta})=\prod_{j=1}^k \pi_{\textrm{obs}}(y_j\mid r, \boldsymbol{\beta})=\prod_{j=1}^k\binom{n_j}{y_j}\frac{B(y_j+rp^E_j, ~n_j-y_j+r(1-p^E_j))}{B(rp^E_j, ~r(1-p^E_j))}.
\end{equation}

The joint posterior density  function of hyper-parameters, $\pi_{\textrm{hyp.post}}(r, \boldsymbol{\beta}\mid \boldsymbol{y})$,  is proportional to their likelihood function in \eqref{marginal} multiplied by the joint hyper-prior PDF in \eqref{hyper-prior}:
\begin{equation}\label{marginal_post}
\pi_{\textrm{hyp.post}}(r, \boldsymbol{\beta}\mid \boldsymbol{y})\propto \pi_{\textrm{hyp.prior}}(r, \boldsymbol{\beta}) \times L(r, \boldsymbol{\beta}).
\end{equation}

Finally, the full posterior density function of $\boldsymbol{p}=(p_1, p_2, \ldots, p_k)^\top$, $r$, and $\boldsymbol{\beta}$ is 
\begin{align}\label{fullpost}
~~~~~\pi_{\textrm{full.post}}(\boldsymbol{p}, r, \boldsymbol{\beta}\mid \boldsymbol{y}) &\propto \pi_{\textrm{hyp.prior}}(r, \boldsymbol{\beta})\times\prod_{j=1}^k \pi_{\textrm{obs}}(y_j\mid p_j)\times \pi_{\textrm{prior}}(p_j\mid r, \boldsymbol{\beta})\\
&\propto \pi_{\textrm{hyp.post}}(r, \boldsymbol{\beta}\mid \boldsymbol{y})\times\prod_{j=1}^k\pi_{\textrm{cond.post}}(p_j\mid r, \boldsymbol{\beta}, \boldsymbol{y})\nonumber
\end{align}
where the distribution for the prior density function of random effect $j$, $\pi_{\textrm{prior}}(p_j\mid r, \boldsymbol{\beta})$, is specified in \eqref{second-beta}, and the distribution of the conditional posterior density of random effect $j$,  $\pi_{\textrm{cond.post}}(p_j\mid r, \boldsymbol{\beta}, \boldsymbol{y})$, is specified in \eqref{post}.

\ssection{POSTERIOR PROPRIETY}\label{sec3}
The full posterior density function  in \eqref{fullpost}  is  proper if and only if $\pi_{\textrm{hyp.post}}(r, \boldsymbol{\beta}\mid \boldsymbol{y})$ is proper  because $\prod_{j=1}^k\pi_{\textrm{cond.post}}(p_j\mid r, \boldsymbol{\beta}, \boldsymbol{y})$ is a product of independent and proper Beta density functions.  We therefore focus on  posterior propriety of $\pi_{\textrm{hyp.post}}(r, \boldsymbol{\beta}\mid \boldsymbol{y})$.

\begin{defn}\label{def1}
Group j whose observed number of successes is neither 0 nor $n_j$, i.e., $1\le y_j\le n_j-1$, is called an interior group. Similarly,  group $j$ is extreme if its observed number of successes is either 0 or $n_j$. The symbol $W_y$ denotes the set of indices corresponding to interior groups, i.e., $W_y\subseteq \{1, 2, \ldots, k\}$, and $k_y$ is the number of interior groups, i.e., the number of indices in $W_y$. We use $W_y^c$ to represent the set of $k-k_y$ indices for extreme groups.  The notation $\boldsymbol{X}\equiv (\boldsymbol{x}_1, \boldsymbol{x}_2, \ldots, \boldsymbol{x}_k)^\top $ refers to the $k\times m$ covariate matrix of all groups ($k\ge m$) and $\boldsymbol{\boldsymbol{X_y}}$ is the $k_y\times m$ covariate matrix of the interior groups.
\end{defn}
The subscript $y$  emphasizes the data-dependence of $k_y$, $W_y$, and $\boldsymbol{\boldsymbol{X_y}}$. The rank of $\boldsymbol{X_y}$ can be smaller than $m$ when $\boldsymbol{X}$ is of full rank $m$ because we obtain $\boldsymbol{X_y}$ by removing rows of extreme groups from $\boldsymbol{X}$. If all groups are interior, then $k_y=k$ and $\boldsymbol{X_y}=\boldsymbol{X}$. If all  groups are extreme, then $k_y=0$ and $\boldsymbol{X_y}$ is not defined.

\subsection{Conditions for posterior propriety}
In Table \ref{posterior_condition1}, we  summarize the necessary and sufficient conditions for posterior propriety according to different hyper-prior PDFs, $f(r)$ and $g(\boldsymbol{\beta})$, under two settings: The data contain at least one interior group ($1\le k_y\le k$) and the data contain only extreme groups ($k_y=0$). 

\begin{table*}[h]
\begin{center}
\caption{Necessary and sufficient conditions for posterior propriety of $\pi_{\textrm{hyp.post}}(r, \boldsymbol{\beta}\mid \boldsymbol{y})$  according to $\pi_{\textrm{hyp.prior}}(r, \boldsymbol{\beta})=f(r)g(\boldsymbol{\beta})$ under two settings: The data contain at least one interior group ($1\le k_y\le k$) and the data contain only extreme groups ($k_y=0$). The notation $I_{\{D\}}$ is the indicator function of $D$. The condition, rank$(\boldsymbol{X_y})=m$, below implicitly requires that $k_y\ge m$ because $\boldsymbol{X_y}$ is a $k_y\times m$ matrix. The condition, $\sum_{j=1}^k I_{\{y_j=n_j\}} \ge 1$, means that the data contain at least one extreme group with all successes. The condition, $\sum_{j=1}^k I_{\{y_j=0\}} \ge 1$, means that the data contain at least one extreme group with all failures.}
\label{posterior_condition1}
\begin{tabular}{|c|c|c|c|}
\hhline{~~--}
\multicolumn{1}{c}{}&& Any proper $f(r)$ & $f(r)\propto 1/r^{u+1}$ for $u>0$ \\
\hline
\multirow{3}{*}{$1\le k_y \le k$} & Any proper $g(\boldsymbol{\beta})$ & Always proper &  iff $k_y\ge u+1$  \\
\hhline{~---}
 &  \multirow{2}{*}{$g(\boldsymbol{\beta})\propto 1$ } & \multirow{2}{*}{iff rank($\boldsymbol{X_y}$)$=m$ } &iff $k_y\ge u+1$ \& \\
&&& rank($\boldsymbol{X_y}$)$=m$\\
 \hline
\multirow{4}{*}{$k_y=0$} & Any proper $g(\boldsymbol{\beta})$ & Always proper &  Never proper \\
\hhline{~---}
& \multirow{3}{*}{$g(\boldsymbol{\beta})\propto 1$} & (When $\boldsymbol{x}_{j}^\top\boldsymbol{\beta}=\beta_1$)  & \multirow{3}{*}{Never proper}\\
&& iff $\sum_{j=1}^k I_{\{y_j=n_j\}} \ge 1$ \&  & \\
&& $\sum_{j=1}^k I_{\{y_j=0\}} \ge 1$ & \\
\hline
\end{tabular}
\end{center}
\end{table*}

To prove these conditions, we divide the first setting ($1\le k_y \le k$) into two: A setting where at least one interior group and at least one extreme group exist ($1\le k_y \le k-1$) and a setting where all groups are interior ($k_y = k$). The key to proving conditions for posterior propriety is to derive certain lower and upper bounds for  $L(r, \boldsymbol{\beta})$ that factor into a function of $r$ and a function of $\boldsymbol{\beta}$. We first derive lower and upper bounds for the Beta-Binomial probability mass function of group $j$  with respect to $r$ and $\boldsymbol{\beta}$ because $L(r, \boldsymbol{\beta})$ is just the product of these probability mass functions of all groups. 

\begin{lem}\label{lemma_bounds}
Lower and upper bounds for the Beta-Binomial probability mass function for interior group $j$ with respect to $r$ and $\boldsymbol{\beta}$ are $rp^E_jq^E_j / (1+r)^{n_j-1}$ and $rp^E_jq^E_j/(1+r)$, respectively, up to a constant multiple.  Those for extreme group $j$ with $y_j=n_j$ are $(p^E_j)^{n_j}$ and $p^E_j$, each, and those for extreme group $j$ with $y_j=0$ are $(q^E_j)^{n_j}$ and $q^E_j$, respectively, up to a constant multiple.
\end{lem}
\begin{proof} 
See Appendix A.
\end{proof}

Lemma \ref{lemma_bounds}  shows that our bounds for the Beta-Binomial probability mass function for either interior or extreme group $j$ with respect to $r$ and $\boldsymbol{\beta}$ factor into a function of $r$ and a function of $\boldsymbol{\beta}$.  Because $L(r, \boldsymbol{\beta})$ is a product of these Beta-Binomial probability mass functions of all groups,  bounds for $L(r, \boldsymbol{\beta})$ also factor into a function of $r$ and a function of $\boldsymbol{\beta}$.  Next we derive certain lower and upper bounds for $L(r, \boldsymbol{\beta})$ with respect to $r$ and $\boldsymbol{\beta}$ under the first setting where all groups are interior. 

\begin{lem}\label{lemma_lik_bound_noextreme}
When all groups are interior ($k_y=k$), $L(r, \boldsymbol{\beta})$ can be bounded by
\begin{equation}\label{upper_lower1}
c_1 \frac{r^k \prod_{j=1}^kp^E_jq^E_j}{(1+r)^{\sum_{j=1}^k(n_j-1)}}\le L(r, \boldsymbol{\beta})\le c_2\frac{r^k\prod_{j=1}^{k} p^E_jq^E_j}{(1+r)^k}
\end{equation}
where $c_1$ and $c_2$ are constants that do not depend on $r$ and $\boldsymbol{\beta}$.
\begin{proof}
See Appendix B.
\end{proof}
\end{lem}


When all groups are interior, the joint posterior density function $\pi_{\textrm{hyp.post}}(r, \boldsymbol{\beta}\mid \boldsymbol{y})$ equipped with any joint hyper-prior PDF $\pi_{\textrm{hyp.prior}}(r, \boldsymbol{\beta})$ is proper if 
\begin{equation}
\int_{\mathbf{R}^m} \int_0^{\infty} \pi_{\textrm{hyp.prior}}(r, \boldsymbol{\beta})\times \frac{r^k\prod_{j=1}^{k} p^E_jq^E_j}{(1+r)^k} drd\boldsymbol{\beta}<\infty
\end{equation}
because $r^k\prod_{j=1}^{k} p^E_jq^E_j/(1+r)^k$ is the upper bound for $L(r, \boldsymbol{\beta})$ specified in \eqref{upper_lower1}. Also, the joint posterior density function $\pi_{\textrm{hyp.post}}(r, \boldsymbol{\beta}\mid \boldsymbol{y})$ is improper if 
\begin{equation}
\int_{\mathbf{R}^m} \int_0^{\infty}  \pi_{\textrm{hyp.prior}}(r, \boldsymbol{\beta})\times\frac{r^k\prod_{j=1}^{k} p^E_jq^E_j}{(1 +r)^{\sum_{j=1}^k(n_j-1)}} drd\boldsymbol{\beta}=\infty
\end{equation}
because $r^k\prod_{j=1}^{k} p^E_jq^E_j/(1+r)^{\sum_{j=1}^k(n_j-1)}$ is the lower bound for $L(r, \boldsymbol{\beta})$ in \eqref{upper_lower1}.


\begin{thm}\label{postprop_noextreme_proper_r}
When all groups are interior  in the data ($k_y=k$), the joint posterior density function of hyper-parameters,  $\pi_{\textrm{hyp.post}}(r, \boldsymbol{\beta}\mid \boldsymbol{y})$, equipped with a proper hyper-prior density function on $r$, $f(r)$, and independently an improper flat hyper-prior density function on $\boldsymbol{\beta}$, $g(\boldsymbol{\beta})\propto 1$, is proper  if and only if $rank(\boldsymbol{X})=m$.
\begin{proof}
See Appendix C.
\end{proof}
\end{thm}

The condition for posterior propriety with a proper hyper-prior distribution for  $r$  is the same as the condition for posterior propriety when $r$ is a completely known constant due to the factorization of the bounds for $L(r, \boldsymbol{\beta})$ in \eqref{upper_lower1}. Thus, the condition for posterior propriety in Theorem \ref{postprop_noextreme_proper_r} arises only from the improper hyper-prior PDF for $\boldsymbol{\beta}$.

\begin{thm}\label{postprop_noextreme_proper_beta}
When all groups are interior  in the data ($k_y=k$), the joint posterior density function of hyper-parameters, $\pi_{\textrm{hyp.post}}(r, \boldsymbol{\beta}\mid \boldsymbol{y})$, equipped with $f(r)\propto 1/r^{u+1}$ for positive $u$ and independently a proper hyper-prior density function on $\boldsymbol{\beta}$, $g(\boldsymbol{\beta})$, is proper if and only if $k\ge u+1$.
\begin{proof}
See Appendix D.
\end{proof}
\end{thm}

The condition for posterior propriety when $\boldsymbol{\beta}$ has a proper hyper-prior distribution is the same as the condition for posterior propriety when $\boldsymbol{\beta}$ is not a parameter to be estimated  ($m=0$) due to the factorization of bounds for $L(r, \boldsymbol{\beta})$ in \eqref{upper_lower1}. Thus, the condition for posterior propriety arises solely from the improper hyper-prior PDF for $r$.

\begin{thm}\label{postprop_noextreme_improper_r_beta}
When all groups are interior in the data ($k_y=k$), the joint posterior density function of hyper-parameters, $\pi_{\textrm{hyp.post}}(r, \boldsymbol{\beta}\mid \boldsymbol{y})$, equipped with the joint hyper-prior density function $\pi_{\textrm{hyp.prior}}(r, \boldsymbol{\beta})\propto 1/r^{u+1}$ for positive $u$  is proper if and only if \emph{(i)} $k\ge u+1$ and  \emph{(ii)} $rank(\boldsymbol{X})=m$.
\end{thm}

\begin{proof} 
See Appendix E.
\end{proof}
The conditions for posterior propriety in Theorem \ref{postprop_noextreme_improper_r_beta} are the combination of the condition in Theorem \ref{postprop_noextreme_proper_r}  and that in Theorem \ref{postprop_noextreme_proper_beta} because of the factorization of bounds for $L(r, \boldsymbol{\beta})$.

We begin discussing the conditions for  posterior propriety under the second setting with at least one interior group and at least one extreme group in the data $(1\le k_y\le k-1)$.

\begin{corollary}\label{cor_ignore_extreme}
With at least one interior group and at least one extreme group  in the data $(1\le k_y\le k-1)$,  posterior propriety is determined solely by interior groups, not by extreme groups.
\begin{proof}
See Appendix F.
\end{proof}
\end{corollary}
Corollary \ref{cor_ignore_extreme} means that  we can remove all the extreme groups from the data to determine posterior propriety, treating the remaining interior groups as a new data set ($k_y=k$). Then we can apply Theorem \ref{postprop_noextreme_proper_r}, \ref{postprop_noextreme_proper_beta}, or \ref{postprop_noextreme_improper_r_beta} to the new data set. If posterior propriety holds with only the interior groups, then posterior propriety with the original data with the combined interior and extreme groups $(1\le k_y\le k-1)$ also holds. Corollary \ref{cor_ignore_extreme}  justifies combining the first and second settings  as shown in Table \ref{posterior_condition1}.

We start by specifying the conditions for posterior propriety under the third setting where there are no interior groups in the data $(k_y=0)$.

\begin{lem}\label{lemma_lik_bound_allextreme}
When all groups are extreme in the data $(k_y=0)$, $L(r, \boldsymbol{\beta})$  can be bounded by
\begin{equation}\label{upper_lower2}
c_3\prod_{j=1}^k (p^E_j)^{n_j\cdot I_{\{y_j =n_j\}}} (q^E_j)^{n_j\cdot I_{\{y_j =0\}}}\le L(r, \boldsymbol{\beta})\le c_4\prod_{j=1}^k (p^E_j)^{I_{\{y_j =n_j\}}}(q^E_j)^{ I_{\{y_j =0\}}}
\end{equation}
where $c_3$ and $c_4$ are constants that do not depend on $r$ and $\boldsymbol{\beta}$.
\begin{proof}
See Appendix G.
\end{proof}
\end{lem}


The  upper bound for $L(r, \boldsymbol{\beta})$ in \eqref{upper_lower2} indicates that the hyper-prior distribution of $r$ must be proper for posterior propriety. If the hyper-prior distribution of $\boldsymbol{\beta}$  is also proper, the resulting posterior is automatically proper. If $g(\boldsymbol{\beta})\propto 1$, then we are not sure about posterior propriety except when only an intercept term is used, i.e., $\boldsymbol{x}_{j}^\top\boldsymbol{\beta}=\beta_1$ for all $j$.

\begin{thm}\label{postprop_allextreme_proper_r}
When all groups are extreme in the data $(k_y=0)$ with an intercept term ($\boldsymbol{x}_{j}^\top\boldsymbol{\beta}=\beta_1$), the posterior density function of hyper-parameters, $\pi_{\textrm{hyp.post}}(r, \boldsymbol{\beta}\mid \boldsymbol{y})$, equipped with a proper hyper-prior density function for $r$, $f(r)$, and independently $g(\beta_1)\propto 1$, is proper  if and only if there are at least one extreme group with all successes  and at least one extreme group with all failures, i.e., ${\sum_{j=1}^kI_{\{y_j=n_j\}}\ge1}$ and $\sum_{j=1}^kI_{\{y_j=0\}}\ge1$.
\begin{proof}
See Appendix H.
\end{proof}
\end{thm}

We leave the conditions for posterior propriety without the restriction that $\boldsymbol{x}_{j}^\top\boldsymbol{\beta}=\beta_1$ for our future research. Posterior propriety is unknown when all groups are extreme in the data ($k_y=0$) with covariate information.

\begin{thm}\label{postprop_allextreme_improper_r}
When all groups are extreme in the data $(k_y=0)$, the posterior density function of hyper-parameters $\pi_{\textrm{hyp.post}}(r, \boldsymbol{\beta}\mid \boldsymbol{y})$, equipped with any improper hyper-prior density function $f(r)$ and independently any  hyper-prior density   $g(\boldsymbol{\beta})$, is always improper. 
\begin{proof}
See Appendix I.
\end{proof}
\end{thm}

\subsection{Posterior propriety in previous studies}\label{post_proper_other_article}
Though the article of \cite{albert1988computational} does not address  posterior propriety for $d\boldsymbol{\beta}dr/(1+r)^2$, when $1\le k_y\le k$  the condition for posterior propriety is that the covariate matrix of interior groups is of full rank $m$, i.e., rank$(\boldsymbol{X_y})=m$. However, when $k_y=0$, posterior propriety is unknown except for a case where only an intercept term is used ($\boldsymbol{x}^\top\boldsymbol{\beta}=\beta_1$), see Table \ref{posterior_condition1}.

The proposition (1c to be specific) in \citet{daniels1999prior} for posterior propriety of the Bayesian BBL  model with the same hyper-prior PDF as \cite{albert1988computational} argues that the posterior distribution is always proper. However, its proof is based on a limited case with only an intercept term, $\boldsymbol{x}_j^\top\boldsymbol{\beta}=\beta_1$. Under this simplified setting, if there is only one extreme group with two trials ($y_1=2$, $n_1=2$), the resulting joint posterior density function of $r$ and $\beta_1$ is 
\begin{equation}\label{proof8_daniels}
\pi_{\textrm{hyp.post}}(r, \beta_1\mid \boldsymbol{y})\propto \frac{(1+rp^E)p^E}{(1+r)^3}.
\end{equation}
The integration of \eqref{proof8_daniels} with respect to $\beta_1$ is not finite because $p^E=\exp(\beta_1)/(1+\exp(\beta_1))$ converges to one as $\beta_1$ approaches infinity.  Table \ref{posterior_condition1}  shows that at least one interior group is required in the data for posterior propriety of the Bayesian BBL model under the simplified setting ($\boldsymbol{x}_j^\top\boldsymbol{\beta}=\beta_1$) of \cite{daniels1999prior}. Moreover, if all groups are extreme in the data  under the simplified setting with an intercept term, the posterior is proper if and only if there exist  at least one extreme group with all successes ($\sum_{j=1}^kI_{\{y_j=n_j\}}\ge1$) and one extreme group with all failures   ($\sum_{j=1}^kI_{\{y_j=0\}}\ge1$) as shown in Table \ref{posterior_condition1}. In our counter-example, there is only one extreme group with all successes, and thus the resulting posterior in \eqref{proof8_daniels} is improper.

With only an intercept term ($\boldsymbol{x}_j^\top\boldsymbol{\beta}=\beta_1$), Chapter 5 of \cite{gelman2013bayesian} specifies that the joint posterior density function $\pi_{\textrm{hyp.post}}(r, \beta_1\mid \boldsymbol{y})$ with $ dr/r^{1.5}$ and independently with the proper standard logistic distribution on $\beta_1$ is proper if there is at least one interior group.  However, the resulting posterior can be improper with this condition. For example, when there is only one interior group with two trials ($y_1=1$, $n_1=2$) with $p^E=\exp(\beta_1)/(1+\exp(\beta_1))=1-q^E$, the joint posterior density function of $r$ and $\beta_1$ is
\begin{equation}\label{proof100}
\pi_{\textrm{hyp.post}}(r, \beta_1\mid \boldsymbol{y})\propto \pi_{\textrm{hyp.prior}}(r, \beta_1)\times L(r, \beta_1)\propto \frac{p^Eq^E}{r^{1.5}}\times\frac{rp^Eq^E}{(1+r)}.
\end{equation}
The integration of this joint posterior density function with respect to $r$ is not finite because the density function goes to infinity as $r$ approaches zero. (The integral of $dr/r^{0.5}$ over the range $[0, 0+\epsilon]$ for a positive constant $\epsilon$ is not finite.) To achieve posterior propriety in this setting, we need  at least two interior groups in the data as shown in Table \ref{posterior_condition1}.

The posterior distributions of \cite{kass1989approximate} and \cite{kahn1996} are always improper regardless of the data due to their hyper-prior PDF $dr/r$. This is because the likelihood function in \eqref{marginal} approaches $c(\boldsymbol{\beta})$, a positive constant with respect to $r$, as $r$ increases to infinity. Then the hyper-prior PDF $dr/r$, whose integration becomes infinite over the range $[\epsilon, \infty)$ for a constant $\epsilon$, governs the right tail behavior of the conditional posterior density function of $r$, $\pi_{\textrm{hyp.cond.post}}(r\mid  \boldsymbol{\beta}, \boldsymbol{y})$. It indicates that $\pi_{\textrm{hyp.cond.post}}(r\mid  \boldsymbol{\beta}, \boldsymbol{y})$ is improper, and thus the joint posterior density $\pi_{\textrm{hyp.post}}(r, \boldsymbol{\beta}\mid \boldsymbol{y})$ is improper. 

\ssection{EXAMPLES}\label{sec4}

\subsection{Data of two bent coins}\label{example_coin}
We have two biased coins; a bent penny and a possibly differently bent nickel ($k=2$). We flip these coins twice for each ($n_1=n_2=2$) and record the number of Heads for the penny ($y_1$) and also for the nickel ($y_2$). We model this experiment as $y_j\mid p_j\sim \textrm{Bin}(2, p_j)$ independently, where $p_j$ is the unknown probability of observing Heads for coin $j$. We assume an i.i.d. prior distribution for random effects, $p_j\mid r, \beta_1\sim \textrm{Beta}(rp^E, rq^E)$, where $p^E=\exp(\beta_1)/[1+\exp(\beta_1)]=1-q^E$, i.e., a BB model.

We look into posterior propriety under four different settings depending on whether the hyper-prior  distribution for  $\beta_1$ (or equivalently $p^E$) is proper or improper flat $d\boldsymbol{\beta}$, and on whether the hyper-prior distribution of $r$ is proper or $dr/r^{2}$. 

\begin{table}[b!]
\caption{The symbol O indicates that the posterior distribution is proper on corresponding data, and the symbol X indicates that the posterior distribution is not proper on corresponding data. }
\label{properpost}
    \begin{subtable}{.5\linewidth}
      \centering
        \caption{Any proper $f(r)$ and any proper $g(\beta_1)$}
          \begin{tabular}{c|ccc}
          
          $y_1 \backslash y_2$ & 0 & 1 & 2 \\
          \hline
          0 & O & O & O \\
          1 & O & O & O \\
          2 & O & O & O \\
          \end{tabular}   
    \end{subtable}%
    \begin{subtable}{.5\linewidth}
      \centering
        \caption{Any proper $f(r)$ and $g(\beta_1)\propto 1$}
          \begin{tabular}{c|ccc}
          $y_1 \backslash y_2$ & 0 & 1 & 2 \\
          \hline
          0 & X & O & O \\
          1 & O & O & O \\
          2 & O & O &  X\\
          \end{tabular}   
    \end{subtable}%

    \begin{subtable}{.5\linewidth}
      \centering
        \caption{$f(r)\propto 1/r^2$ and any proper $g(\beta_1)$}
          \begin{tabular}{c|ccc}
          $y_1 \backslash y_2$ & 0 & 1 & 2 \\
          \hline
          0 & X & X & X \\
          1 & X & O &  X\\
          2 & X & X& X \\
          \end{tabular}   
    \end{subtable}%
    \begin{subtable}{.5\linewidth}
      \centering
        \caption{$f(r)\propto 1/r^2$ and $g(\beta_1)\propto1$}
          \begin{tabular}{c|ccc}
          $y_1 \backslash y_2$ & 0 & 1 & 2 \\
          \hline
          0 & X & X & X \\
          1 & X & O &  X\\
          2 & X & X& X \\
          \end{tabular}   
    \end{subtable}%

\end{table}

Table \ref{properpost} shows when the posterior distribution is proper (denoted by O) and when it is not (denoted by X). The posterior distribution in case (a) is always proper because both hyper-prior distributions for $r$ and $\beta_1$ are  proper. In case (b) where $\beta_1$ has the Lebesque measure and $r$ has a proper hyper-prior PDF, the posterior is proper unless both coins land either all Heads or all Tails. This is because the condition for posterior propriety is that the covariate matrix of interior coins is of full rank and this condition without any covariates is met if at least one coin is interior; see Table \ref{posterior_condition1}. In cases (c) and (d), where  $r$ has the improper hyper-prior PDF, $dr/r^2$, posterior propriety holds only when each coin shows one Head and one Tail, i.e., both coins are interior ($y_1=y_2=1$); see Table \ref{posterior_condition1}.  Cases (c) and (d) have the same condition for posterior propriety because the condition that arises from the improper flat hyper-prior PDF for $\beta_1$ in case (d) is automatically met if the condition arising from the improper hyper-prior PDF for $r$, i.e., $k_y\ge 2$, is met.

\subsection{Data of five hospitals}\label{example_hospital}

\begin{table*}[b!]
 \caption{Data of five hospitals. The number of patients in hospital $j$ is denoted by $n_j$, the number of death in hospital $j$ is denoted by $y_j$, and the expected mortality rate for hospital $j$ is denoted by EMR$_j$.}
 \small
\label{hospitaldata}
\begin{center}
\begin{tabular}{cccccccccccccccccccc}
\hline
Hospital $j$ & 1  & 2 &  3 & 4 & 5\\
\hline
$n_j$ & 54  & 75 &  93 & 104 & 105\\
$y_j$ & 3  & 4 & 1 & 1 & 1\\
EMR$_j$ (\%) & 4.30  & 2.21 & 2.59  & 4.73 & 3.28 \\
\hline
\end{tabular}
\end{center}
\end{table*}

New York State Cardiac Advisory Committee (2014) has reported the outcomes for the Valve Only and Valve/CABG surgeries. The data are based on the patients discharged between December 1, 2008, and November 30, 2011 in 40 non-federal hospitals in New York State. We select the smallest five hospitals with respect to the number of patients for simplicity. Table \ref{hospitaldata} shows the data including the number of cases ($n_j$), the number of deaths ($y_j$), and expected mortality rate (EMR$_j$). The EMR$_j$ is a hospital-wise average over the predicted probabilities of death for each patient; the larger the EMR$_j$ is, the more difficult cases hospital $j$ handles. We use the EMR$_j$ as a continuous covariate. We assume $y_j\mid p_j\stackrel{\textrm{indep.}}{\sim}\textrm{Bin}(n_j, p_j)$ independently. We also assume that the unknown true mortality rates $p_j$ come from independent conjugate Beta prior distributions in \eqref{second-beta} with $x_j^T\boldsymbol{\beta}=\beta_1x_{1j}+\beta_1x_{2j}$, where $x_{1j}=1$ and $x_{2j}=\textrm{EMR}_j$. 

\begin{table*}[b!]
 \caption{Two hypothetical data sets of five hospitals. The number of patients in hospital $j$ is denoted by $n_j$, the number of death in hospital $j$ is denoted by $y_j$, and the expected mortality rate for hospital $j$ is denoted by EMR$_j$. In the first data set, only one hospital is interior. In the second data set, two hospitals are interior but their EMRs are the same.}
\label{hospitaldata2}
\small\begin{center}
\begin{tabular}{cccccccccccccccccccc}
\hline
Hospital $j$ & 1  & 2 &  3 & 4 & 5\\
\hline
$n_j$ & 54  & 75 &  93 & 104 & 105\\
$y_j$ & 1  & 0 & 0 & 0 & 0\\
EMR$_j$ (\%) & 4.30  & 2.21 & 2.59  & 4.73 & 3.28 \\
\hline
\end{tabular}~~~
\begin{tabular}{cccccccccccccccccccc}
\hline
Hospital $j$ & 1  & 2 &  3 & 4 & 5\\
\hline
$n_j$ & 54  & 75 &  93 & 104 & 105\\
$y_j$ & 1  & 2 & 0 & 0 & 0\\
EMR$_j$ (\%) & 4.30  & 4.30 & 2.59  & 4.73 & 3.28 \\
\hline
\end{tabular}
\end{center}
\end{table*}

We consider  four joint hyper-prior densities:  $d\boldsymbol{\beta}dr/r^2$, $d\boldsymbol{\beta}dr/(1+r)^2$, $d\boldsymbol{\beta}dr/r^{1.5}$ and $d\boldsymbol{\beta}dr/(1+r)^{1.5}$. The conditions for posterior propriety are the same for all four hyper-prior PDFs; $k_y\ge2$ and rank($\boldsymbol{X_y})=2$. However, the latter condition automatically meets the former condition, and thus the condition for posterior propriety is simply that the covariate matrix of interior hospitals is of full rank. The data in Table \ref{hospitaldata} satisfy the condition for posterior propriety because all the hospitals are interior ($0<y_j<n_j$ for all $j$ and thus $k=k_y=5$) and their covariate matrix $\boldsymbol{X}=\boldsymbol{X_y}$ is of full rank. 

Based on the data in Table \ref{hospitaldata}, we make two hypothetical data sets in Table \ref{hospitaldata2}. In the first hypothetical data set, only one hospital is interior. The resulting posterior distribution is improper for the four joint hyper-prior densities because the rank of the covariate matrix of this interior hospital is one (rank$(\boldsymbol{X_y})=1\neq2$). In the second hypothetical data set, two hospitals are interior but their EMRs are the same, meaning that the rank of the covariate matrix of these two interior hospitals is one again. Thus, the resulting posterior is improper for the four joint hyper-prior densities.



\ssection{CONCLUSION}
The Beta-Binomial-Logit (BBL) model accounts for the overdispersion in the Binomial data obtained from several independent groups with their covariate information considered. From a Bayesian perspective, we derive data-dependent necessary and sufficient conditions for posterior propriety of the Bayesian BBL model equipped with a joint hyper-prior PDF, $g(\boldsymbol{\beta})/(t+r)^{u+1}$, where $t\ge0$, $u>0$, and $g(\boldsymbol{\beta})$ can be any proper PDF or an improper flat density in Table \ref{posterior_condition1}. This joint hyper-prior PDF encompasses those used in the literature.

There are several opportunities to build upon our work. The data-dependent conditions for posterior propriety make it hard to evaluate frequency properties of the Bayesian BBL model  because the model does not define a frequency procedure for all possible data sets; the resulting posterior can be improper or unknown for some possible data sets. Thus, in a repeated sampling simulation, we may evaluate frequency properties given only the simulated data sets that achieve posterior propriety. Also, posterior propriety of the Bayesian BBL model is unknown when all groups are extreme in the data with some covariates. We leave these for  our future research.

\bigskip
\begin{appendices}

\section{\!\!\!\!\!:~ PROOF OF LEMMA \ref{lemma_bounds}}\label{proof_lemma3_1} 
If group $j$ is interior ($1\le y_j \le n_j-1$, $n_j\ge2$), we can derive an upper bound for the Beta-Binomial probability mass function of interior group $j$ with respect to $r$ and $\boldsymbol{\beta}$ as follows. (All bounds in this proof are up to a constant multiple.) With notation  $q^E_j=1-p^E_j$,
\begin{align}
\pi_{\textrm{obs}}(y_j\mid r, \boldsymbol{\beta}) &\propto \frac{B(y_{j}+rp^E_j,~ n_{j}-y_{j}+rq^E_j)}{B(rp^E_j,~ rq^E_j)}\label{proof1}\\
&= \frac{B(1+rp^E_j,~1+rq^E_j)}{B(rp^E_j,~ rq^E_j)}\frac{B(y_{j}+rp^E_j,~ n_{j}-y_{j}+rq^E_j)}{B(1+rp^E_j,~ 1+rq^E_j)}\label{proof1_2}\\
&= \frac{rp^E_jq^E_j}{1+r}\frac{B(y_{j}+rp^E_j,~ n_{j}-y_{j}+rq^E_j)}{B(1+rp^E_j,~ 1+rq^E_j)}\label{proof2}\\
&= \frac{rp^E_jq^E_j}{1+r}\frac{\int_0^1 v^{y_{j}-1+rp^E_j}(1-v)^{n_{j}-y_{j}-1+rq^E_j}dv}{\int_0^1 v^{rp^E_j}(1-v)^{rq^E_j}dv}~\le~ \frac{rp^E_jq^E_j}{1+r}.\label{proof2_2}
\end{align}
The ratio of the two beta functions in \eqref{proof2_2}  is less than or equal to one because the integrand of the beta function in the numerator is less than or equal to the integrand of the beta function in the denominator, considering that $0\le y_j-1\le n_j-2$ and $0\le n_j-y_j-1\le n_j-2$.

A lower bound for the ratio of the two beta functions in  \eqref{proof1} is
\begin{eqnarray}
&&~~~\frac{B(y_{j}+rp^E_j,~ n_{j}-y_{j}+r(1-p^E_j))}{B(rp^E_j,~ r(1-p^E_j))}\label{proof3}\\
&&~~~=\frac{(y_j-1+rp^E_j)\cdots(1+rp^E_j)rp^E_j(n_j-y_j-1+rq^E_j)\cdots(1+rq^E_j)rq^E_j}{(n_j-1+r)(n_j-2+r)\cdots(1+r)r}\label{proof4}\\
&& ~~~\ge\frac{r^2p^E_jq^E_j}{(n_j-1+r)(n_j-2+r)\cdots(1+r)r}\ge \frac{rp^E_jq^E_j}{(n_{\textrm{max}} +r)^{n_j-1}}\ge \frac{rp^E_jq^E_j}{(1+r)^{n_j-1}}\label{proof5}
\end{eqnarray}
where $n_{\textrm{max}}\equiv \{n_1, n_2, \ldots, n_k\}$. The first inequality in \eqref{proof5} holds because each factor except $rp^E_j$ and $rq^E_j$ in the numerator of \eqref{proof4} is greater than or equal to one. The third inequality holds up to a constant multiplication, $1/n_j^{n_j-1}$, because $(n_j+r)/(1+r)\le n_j$.

If group $j$ is extreme with  all successes ($y_j=n_j\ge1$),  the upper bound for the Beta-Binomial probability mass function of group $j$ with respect to $r$ and $\boldsymbol{\beta}$ is
\begin{equation}
\pi_{obs}(y_j=n_j\mid r, \boldsymbol{\beta}) \propto \frac{B(n_{j}+rp^E_j,~ rq^E_j)}{B(rp^E_j,~ rq^E_j)}~\le~\frac{B(1+rp^E_j,~ rq^E_j)}{B(rp^E_j,~ rq^E_j)}=p^E_j.\label{proof6}
\end{equation}
The inequality holds because  the integrand of the beta function in the numerator becomes the largest when $n_j=1$. The lower bound for the Beta-Binomial probability mass function of this extreme group with respect to $r$ and $\boldsymbol{\beta}$ is  
\begin{equation}
\frac{B(n_{j}+rp^E_j,~ rq^E_j)}{B(rp^E_j,~ rq^E_j)}=\frac{(n_j-1+rp^E_j)(n_j-2+rp^E_j)\cdots(1+rp^E_j)p^E_j}{(n_j-1+r)(n_j-2+r)\cdots(1+r)}\ge (p^E_j)^{n_j}.\label{proof7}
\end{equation}
The inequality holds because the ratio of the two beta functions in \eqref{proof7} is a decreasing function of $r$, and thus the lower bound is achieved as $r$ goes to infinity.

Similarly, when group $j$ is extreme with all failures ($y_j=0,$ $n_j\ge1$), we can bound the ratio of the two beta functions of this extreme group  by
\begin{equation}\label{bound_extreme}
( q^E_j)^{n_j}< \frac{B(rp^E_j,~ n_{j}+rq^E_j)}{B(rp^E_j,~ rq^E_j)}<q^E_j.
\end{equation}

\section{\!\!\!\!\!:~ PROOF OF LEMMA \ref{lemma_lik_bound_noextreme}}\label{proof_lemma3_2}
Without any extreme groups in the data, an upper bound for $L(r, \boldsymbol{\beta})$ is the product of the $k$ upper bounds for the Beta-Binomial probability mass function of each interior group in  \eqref{proof2}, i.e., $r^k(\prod_{j=1}^{k} p^E_jq^E_j)/(1+r)^k$. Similarly, a lower bound for $L(r, \boldsymbol{\beta})$ is the product of the $k$ lower bounds for the Beta-Binomial probability mass function of each interior group in  \eqref{proof5}, i.e., $r^k (\prod_{j=1}^kp^E_jq^E_j)/(1+r)^{\sum_{j=1}^k (n_j-1)}$. It is clear that both bounds factor into a function of $r$ and a function of $\boldsymbol{\beta}$.

\section{\!\!\!\!\!:~ PROOF OF THEOREM \ref{postprop_noextreme_proper_r}}\label{proof_theorem3_1}
Because the $r$ part of the  upper bound for $L(r, \boldsymbol{\beta})$  in Lemma \ref{lemma_lik_bound_noextreme}, i.e., $r^k/(1+r)^k$,  is always less than one, an upper bound for $\pi_{\textrm{hyp.post}}(r, \boldsymbol{\beta}\mid \boldsymbol{y})$, up to a normalizing constant, factors into a function of $r$ and a function of $\boldsymbol{\beta}$ as follows:
\begin{equation}\label{proof8}
\pi_{\textrm{hyp.post}}(r, \boldsymbol{\beta}\mid \boldsymbol{y})\propto f(r) g(\boldsymbol{\beta}) L(r, \boldsymbol{\beta})<f(r) \cdot\prod_{j= 1}^k p^E_jq^E_j.
\end{equation}
The integration of $f(r)$ with respect to $r$ is finite because it is a proper hyper-prior PDF.  The integration of $\prod_{j= 1}^k p^E_jq^E_j$ with respect to $\boldsymbol{\beta}$ is finite if and only if the covariate matrix of all groups, $\boldsymbol{X}$,  is of full rank $m$. To show the sufficient condition, let us choose $m$ sub-groups, whose index set is denoted by $W_{\textrm{sub}}$, such that the $m\times m$ covariate matrix of the sub-groups is still of full rank $m$. Then,
\begin{equation}\label{proof9}
\prod_{j= 1}^k p^E_jq^E_j<\prod_{j\in W_{\textrm{sub}}} p^E_jq^E_j=\prod_{j\in W_{\textrm{sub}}} \frac{\exp(\boldsymbol{x}_j^\top\boldsymbol{\beta})}{[1+\exp(\boldsymbol{x}_j^\top\boldsymbol{\beta})]^2}.
\end{equation}
The integration of this upper bound in \eqref{proof9} with respect to $\boldsymbol{\beta}$ factors into $m$ separate integrations after linear transformations, $h_j =\boldsymbol{x}_j^\top\boldsymbol{\beta}$     for all $j\in  W_{\textrm{sub}}$, whose Jacobian is a constant:
\begin{equation}\label{proof10}
\int_{R^m}\prod_{j\in W_{\textrm{sub}}} \frac{\exp(\boldsymbol{x}_j^\top\boldsymbol{\beta})}{[1+\exp(\boldsymbol{x}_j^\top\boldsymbol{\beta})]^2}d\boldsymbol{\beta} \propto \int_{-\infty}^{\infty}\!\!\cdots\!\!\int_{-\infty}^{\infty}\prod_{j\in W_{\textrm{sub}}} \frac{\exp(h_j)}{[1+\exp(h_j)]^2}dh_j=1.
\end{equation}
Each integration on the right hand side leads to one because each integrand is a proper  density function of the standard logistic distribution with respect to $h_j$.

Next, we show that if the rank of $\boldsymbol{X}$ is not of full rank $m$, then the integration of  the $\boldsymbol{\beta}$ part of the  lower bound for $L(r, \boldsymbol{\beta})$ in Lemma \ref{lemma_lik_bound_noextreme}, i.e.,  $\prod_{j= 1}^k p^E_jq^E_j$, cannot be finite. Without loss of generality, let us assume that the rank of $\boldsymbol{X}$ is $m-1$ and that the last column of $\boldsymbol{X}$ can be expressed as a linear function of the first $m-1$ columns.  Due to the singularity of $\boldsymbol{X}$, we can always find $m-1$ linear functions, $t_i(\beta_i, \beta_m)$, $i=1, 2, \ldots, m-1$, such that $\boldsymbol{x}_j^\top\boldsymbol{\beta}=x_{j1}t_1(\beta_1, \beta_m)+x_{j2}t_2(\beta_2, \beta_m)+\cdots +x_{j, m-1}t_{m-1}(\beta_{m-1}, \beta_m)$. As a result, the integration of $\prod_{j= 1}^k p^E_jq^E_j$ with respect to $\boldsymbol{\beta}$ is infinity after a linear transformation from $\boldsymbol{\beta}$ to $(\beta_1^{\ast}=t_1(\beta_1, \beta_m),~ \beta_2^{\ast}=t_2(\beta_2, \beta_m),~ \ldots,~ \beta_{m-1}^{\ast}=t_{m-1}(\beta_{m-1}, \beta_m),~ \beta_m)^\top$, whose Jacobian is one. For notational simplicity, we use two $(m-1)\times 1$ vectors, $\boldsymbol{x}_j^\ast\equiv (x_1, x_2, \ldots, x_{m-1})^\top$ and $\boldsymbol{\beta}^\ast=(\beta_1^\ast, \beta_2^\ast, \ldots, \beta_{m-1}^\ast)^\top$:
\begin{equation}\label{proof11}
\int_{R^m}\prod_{j=1}^k\frac{\exp(\boldsymbol{x}_j^\top\boldsymbol{\beta})}{[1+\exp(\boldsymbol{x}_j^\top\boldsymbol{\beta})]^2}d\boldsymbol{\beta}=\int_{R^{m-1}}\prod_{j=1}^k\frac{\exp(\boldsymbol{x}_j^{\ast T}\boldsymbol{\beta}^\ast)}{[1+\exp(\boldsymbol{x}_j^{\ast T}\boldsymbol{\beta}^\ast)]^2}d\boldsymbol{\beta}^\ast\times\int_R d\beta_m,
\end{equation}
where $\int_R d\beta_m=\infty$.

\section{\!\!\!\!\!:~ PROOF OF THEOREM \ref{postprop_noextreme_proper_beta}}\label{proof_theorem3_2}

The $\boldsymbol{\beta}$ part of the  upper bound for $L(r, \boldsymbol{\beta})$ in Lemma \ref{lemma_lik_bound_noextreme}, i.e., $\prod_{j=1}^k p^E_jq^E_j$, is always less than one. Thus, the upper bound for $\pi_{\textrm{hyp.post}}(r, \boldsymbol{\beta}\mid \boldsymbol{y})$ up to a normalizing constant factors into a function of $r$ and a function of $\boldsymbol{\beta}$ as follows:
\begin{equation}\label{proof12}
\pi_{\textrm{hyp.post}}(r, \boldsymbol{\beta}\mid \boldsymbol{y})\propto f(r) g(\boldsymbol{\beta}) L(r, \boldsymbol{\beta})<\frac{r^{k - (u +1)}g(\boldsymbol{\beta})}{(1+r)^{k}} .
\end{equation}
The integration of this upper bound with respect to $r$ is finite if $k\ge u+1$ because in this case we can bound the $r$ part by $1/(1+r)^{u+1}$ whose integration with respect to $r$ is always finite. The integration of $g(\boldsymbol{\beta})$ with respect to $\boldsymbol{\beta}$ is finite because $g(\boldsymbol{\beta})$  is a proper probability density function. 

If $k< u+1$, then the integration of the lower bound for $\pi_{\textrm{hyp.post}}(r, \boldsymbol{\beta}\mid \boldsymbol{y})$ is not finite because there is $r^{k}$ in the numerator of the  lower bound for $L(r, \boldsymbol{\beta})$  in Lemma \ref{lemma_lik_bound_noextreme}. Specifically, once  multiplying $f(r)~(\propto dr/r^{u+1})$ by $r^{k}$, we know that $r^{k-(u+1)}$ goes to infinity as $r$ approaches zero if $k< u+1$.

\section{\!\!\!\!\!:~ PROOF OF THEOREM \ref{postprop_noextreme_improper_r_beta}}\label{proof_theorem3_3}

Based on the  upper bound for $L(r, \boldsymbol{\beta})$ in Lemma \ref{lemma_lik_bound_noextreme}, the upper bound for $\pi_{\textrm{hyp.post}}(r, \boldsymbol{\beta}\mid \boldsymbol{y})$ up to a normalizing constant factors into a function of $r$ and a function of $\boldsymbol{\beta}$ as follows:
\begin{equation}\label{proof13}
\pi_{\textrm{hyp.post}}(r, \boldsymbol{\beta}\mid \boldsymbol{y})\propto \pi_{\textrm{hyp.prior}}(r, \boldsymbol{\beta}) L(r, \boldsymbol{\beta})<\frac{r^{k - (u +1)}}{(r+1)^{k}}\prod_{j=1}^{k} p^E_jq^E_j.
\end{equation}
The double integration on the upper bound in  \eqref{proof13} with respect to $r$ and $\boldsymbol{\beta}$ is finite if and only if $(i)$ $k \ge u +1$ for the $r$ part as proved in Theorem \ref{postprop_noextreme_proper_beta} and $(ii)$ the $k\times m$ covariate matrix of all groups $\boldsymbol{X}$ has a full rank $m$ for the $\boldsymbol{\beta}$ part as proved in Theorem \ref{postprop_noextreme_proper_r}.

If at least one condition is not met, then $\pi_{\textrm{hyp.post}}(r, \boldsymbol{\beta}\mid \boldsymbol{y})$ becomes improper as proved in Theorem \ref{postprop_noextreme_proper_r} and \ref{postprop_noextreme_proper_beta}.

\section{\!\!\!\!\!:~ PROOF OF COROLLARY \ref{cor_ignore_extreme}}\label{proof_corollary3_1}
Regarding the sufficient conditions for posterior propriety, an upper bound for $L(r, \boldsymbol{\beta})$ up to a constant multiplication is
\begin{align}
L(r, \boldsymbol{\beta}) &\propto\prod^{k}_{j=1} \frac{B(y_{j}+rp^E_j,~ n_{j}-y_{j}+rq^E_j)}{B(rp^E_j,~ rq^E_j)}< \prod_{j\in W_y}   \frac{B(y_{j}+rp^E_j,~ n_{j}-y_{j}+rq^E_j)}{B(rp^E_j,~ rq^E_j)}\label{cor1}\\
&=\prod_{j\in W_y} \frac{rp^E_jq^E_j}{1+r}\frac{B(y_{j}+rp^E_j,~ n_{j}-y_{j}+rq^E_j)}{B(1+rp^E_j,~ 1+rq^E_j)}\le\frac{r^{k_y}\prod_{j\in W_y} p^E_jq^E_j}{(1+r)^{k_y}}.\label{cor3}
\end{align}
The  inequality in \eqref{cor1} holds because the  upper bound for the ratio of two beta functions for extreme group $j$  is either $p^E_j$ ($<1$) in \eqref{proof6} or $q^E_j$ ($<1$) in  \eqref{bound_extreme}. The  inequality in \eqref{cor3} holds because  the integrand of the beta function in the numerator  is less than or equal to the integrand of the beta function in the denominator. 

The upper bound for $L(r, \boldsymbol{\beta})$ in  \eqref{cor3} would be the same as the  upper bound for $L(r, \boldsymbol{\beta})$ in Lemma~\ref{lemma_lik_bound_noextreme} if we removed all extreme groups from the data and treated the interior groups as a new data set ($k_y=k$). Thus, if the joint posterior density function $\pi_{\textrm{hyp.post}}(r, \boldsymbol{\beta}\mid \boldsymbol{y})$ is proper with the new data set of $k_y$ interior groups based on Theorem \ref{postprop_noextreme_proper_r}, \ref{postprop_noextreme_proper_beta}, or \ref{postprop_noextreme_improper_r_beta}, then posterior propriety with the original data with all interior and all extreme groups combined $(1\le k_y\le k-1)$ also holds. In other words, the extreme groups do not affect the sufficient condition for posterior propriety at all no matter how many of them are in the data as long as there exists at least one interior group in the data.

For the necessary conditions for posterior propriety, we will show that if a new data set with all the extreme groups removed does not meet the conditions for posterior propriety based on Theorem \ref{postprop_noextreme_proper_r}, \ref{postprop_noextreme_proper_beta}, or \ref{postprop_noextreme_improper_r_beta}, then $\pi_{\textrm{hyp.post}}(r, \boldsymbol{\beta}\mid \boldsymbol{y})$ is still improper even after we add extreme groups into the new data. 

Because a  lower bound for the Beta-Binomial probability mass function for extreme group $j$ is either $(p^E_j)^{n_j}$ in  \eqref{proof7} or $(q^E_j)^{n_j}$ in  \eqref{bound_extreme}, the extra product term  for  extreme groups to the lower bound for the likelihood function based only on interior groups is $\prod_{i \in W^c_y}(p^E_{i})^{n_iI_{\{y_i=n_i\}}}(q^E_{i})^{n_iI_{\{y_i=0\}}}$. 

Specifically, let us  consider a proper hyper-prior PDF for $r$, $f(r)$, and an improper flat  hyper-prior PDF for $\boldsymbol{\beta}$, $g(\boldsymbol{\beta})\propto d\boldsymbol{\beta}$ as in Theorem \ref{postprop_noextreme_proper_r}. Suppose  we removed all the extreme groups in the data. If the rank of $\boldsymbol{X_y}$ is not of full rank, e.g., rank$(\boldsymbol{X_y})=m-1$, then we  see the term $\int_R d\beta_m$ in \eqref{proof11}. This term does not disappear even after we add all the extreme groups to the data because multiplying $\prod_{i \in W^c_y}(p^E_{i})^{n_iI_{\{y_i=n_i\}}}(q^E_{i})^{n_iI_{\{y_i=0\}}}$ by the first integrand in \eqref{proof11} cannot make the term, $\int_R d\beta_m$, disappear. It means that $\pi_{\textrm{hyp.post}}(r, \boldsymbol{\beta}\mid \boldsymbol{y})$ is still improper.

Next, we consider $f(r)\propto dr/r^{u+1}$ for positive $u$ and a proper hyper-prior PDF on $\boldsymbol{\beta}$, $g(\boldsymbol{\beta})$ as in Theorem \ref{postprop_noextreme_proper_beta}. Because contribution of extreme groups to the lower bound for the likelihood function, i.e., $\prod_{i \in W^c_y}(p^E_{i})^{n_iI_{\{y_i=n_i\}}}(q^E_{i})^{n_iI_{\{y_i=0\}}}$,  is free of $r$, if $k_y$ is smaller than $u+1$, then $\pi_{\textrm{hyp.post}}(r, \boldsymbol{\beta}\mid \boldsymbol{y})$ is still improper even after we add all the extreme groups into the data.

If the data of interior groups do not meet the condition for posterior propriety specified in Theorem \ref{postprop_noextreme_improper_r_beta}, then adding the extreme groups cannot change the result of posterior propriety. This is because Theorem \ref{postprop_noextreme_improper_r_beta} is an improper mixture of Theorem \ref{postprop_noextreme_proper_r} and \ref{postprop_noextreme_proper_beta} and we already showed that extreme groups can be ignored in determining posterior propriety in Theorem \ref{postprop_noextreme_proper_r} and \ref{postprop_noextreme_proper_beta}.

\section{\!\!\!\!\!:~ PROOF OF LEMMA \ref{lemma_lik_bound_allextreme}}\label{proof_lemma3_3}

A  lower bound for the Beta-Binomial probability mass function of extreme group $j$ is either $(p^E_j)^{n_j}$ in \eqref{proof7} or $(q^E_j)^{n_j}$ in  \eqref{bound_extreme} depending on whether $y_j=n_j$ or $y_j=0$. Thus, the product of $k$  lower bounds for the Beta-Binomial probability mass functions of extreme groups, i.e., $\prod_{j=1}^k (p^E_j)^{n_j\cdot I_{\{y_j =n_j\}}}(q^E_j)^{n_j\cdot I_{\{y_j =0\}}}$, bounds $L(r, \boldsymbol{\beta})$ from below.

The product of the $k$  upper bounds for the Beta-Binomial probability mass functions of extreme groups in \eqref{proof6} or \eqref{bound_extreme}, i.e., $\prod_{j=1}^k (p^E_j)^{I_{\{y_j =n_j\}}} (q^E_j)^{I_{\{y_j =0\}}}$, bounds $L(r, \boldsymbol{\beta})$ from above.

\section{\!\!\!\!\!:~ PROOF OF THEOREM \ref{postprop_allextreme_proper_r}}\label{proof_theorem3_4}
With $\boldsymbol{x}_j^\top\boldsymbol{\beta}=\beta_1$ for all $j$, we know that $p^E_j=p^E=1-q^E=\exp(\beta_1)/(1+\exp(\beta_1))$. Considering the  upper bound for $L(r, \beta_1)$ specified in Lemma \ref{lemma_lik_bound_allextreme},  an upper bound for $\pi_{\textrm{hyp.post}}(r, \beta_1\mid \boldsymbol{y})$ up to a normalizing constant factors into a function of $r$ and a function of $\beta_1$ as follows:
\begin{equation}\label{proof21}
\pi_{\textrm{hyp.post}}(r, \beta_1\mid \boldsymbol{y})\propto f(r) g(\beta_1) L(r, \beta_1)<f(r)  (p^E)^{\sum_{j=1}^kI_{\{y_j =n_j\}}}(q^E)^{\sum_{j=1}^k I_{\{y_j =0\}}}.
\end{equation}
This upper bound in  \eqref{proof21} can  be bounded one more time  from above by $f(r)p^Eq^E$ because $\sum_{j=1}^kI_{\{y_j=n_j\}}\ge1$ and $\sum_{j=1}^kI_{\{y_j=0\}}\ge1$. The integration of $f(r)$ with respect to $r$ is finite because $f(r)$  is proper.  The integration of $p^Eq^E$ with respect to $\beta_1$ is finite because $p^Eq^E$ is the  density function of the standard logistic distribution with respect to $\beta_1$.

For the necessary condition, if all the extreme groups have only successes ($y_j=n_j$ for all $j$), then we can bound $\pi_{\textrm{hyp.post}}(r, \beta_1\mid \boldsymbol{y})$ from below  using the lower bound in Lemma \ref{lemma_lik_bound_allextreme} up to a normalizing constant as follows:
\begin{equation}\label{proof22}
\pi_{\textrm{hyp.post}}(r, \beta_1\mid \boldsymbol{y})\propto f(r) g(\beta_1) L(r, \beta_1)>f(r)  (p^E)^{\sum_{j=1}^k n_j}.
\end{equation}
The integration of this lower bound in \eqref{proof22} with respect to $\beta_1$ is not finite because $p^E=\exp(\beta_1)/(1+\exp(\beta_1))$ converges to one as $\beta_1$ approaches infinity. Similarly, ${\pi_{\textrm{hyp.post}}(r, \beta_1\mid \boldsymbol{y})}$ is improper if all the extreme groups have only failures ($y_j=0$ for all $j$).

\section{\!\!\!\!\!:~ PROOF OF THEOREM \ref{postprop_allextreme_improper_r}}\label{proof_theorem3_5}

Because the  lower bound for $L(r, \boldsymbol{\beta})$ in Lemma \ref{lemma_lik_bound_allextreme} is free of $r$, $L(r, \boldsymbol{\beta})$ cannot make  the integration of $f(r)$ finite when $f(r)$ is improper. Thus, $\pi_{\textrm{hyp.post}}(r, \boldsymbol{\beta}\mid \boldsymbol{y})$ should always be improper under this setting.

\end{appendices}
\bibliography{bibliography}
\bibliographystyle{apalike}
\end{document}